\documentclass[11pt]{amsart}
\usepackage{amssymb,amsmath,amsfonts,amscd,euscript}

\newcommand{\nc}{\newcommand}

\numberwithin{equation}{section}
\newtheorem{thm}{Theorem}[section]
\newtheorem{prop}[thm]{Proposition}
\newtheorem{lem}[thm]{Lemma}
\newtheorem{cor}[thm]{Corollary}
\theoremstyle{remark}
\newtheorem{rem}[thm]{Remark}

\newtheorem{example}[thm]{Example}
\newtheorem{dfn}[thm]{Definition}

\nc{\gl}{\mathfrak{gl}}
\nc{\GL}{\mathfrak{GL}}
\nc{\g}{\mathfrak{g}}
\nc{\gh}{\widehat\g}
\nc{\h}{\mathfrak{h}}
\nc{\la}{\lambda}
\nc{\al}{\alpha }
\nc{\be}{\beta }
\nc{\ve}{\varepsilon }
\nc{\om}{\omega }

\nc{\ta}{\theta}
\nc{\veps}{\varepsilon}
\nc{\ch}{{\mathop {\rm ch}}}
\nc{\Tr}{{\mathop {\rm Tr}\,}}
\nc{\Id}{{\mathop {\rm Id}}}
\nc{\ad}{{\mathop {\rm ad}}}
\nc{\bra}{\langle}
\nc{\ket}{\rangle}
\nc{\x}{{\bf x}}
\nc{\bs}{{\bf s}}
\nc{\bp}{{\bf p}}
\nc{\bc}{{\bf c}}
\nc{\pa}{\partial}
\nc{\ld}{\ldots}
\nc{\cd}{\cdots}
\nc{\hk}{\hookrightarrow}
\nc{\T}{\otimes}
\newcommand{\bea}{\begin{equation}}
\newcommand{\ena}{\end{equation}}
\nc{\gr}{\mathrm{gr}}
\nc{\ov}{\overline}

\nc{\cO}{\mathcal O}
\nc{\cF}{\mathcal F}
\nc{\cL}{\mathcal L}
\nc{\msl}{\mathfrak{sl}}
\nc{\mgl}{\mathfrak{gl}}
\nc{\U}{\mathrm U}
\nc{\V}{\EuScript V}
\nc{\bH}{\EuScript H}
\nc{\Res}{\mathrm{Res\ }}

\newcommand{\bC}{{\mathbb C}}

\newcommand{\bN}{{\mathbb N}}
\newcommand{\bP}{{\mathbb P}}
\newcommand{\bG}{{\mathbb G}}

\newcommand{\fb}{{\mathfrak b}}

\newcommand{\fn}{{\mathfrak n}}

\newcommand{\Fl}{\EuScript{F}}

\newcommand{\spa}{\mathrm{span}}
\newcommand{\ol}{\overline}
\newcommand{\bS}{{\bf S}}
\newcommand{\bV}{{\bf V}}
\newcommand{\bd}{{\bf d}}

\begin{document}

\title[Degenerate flag varieties of type A: BW theorem]
{Degenerate flag varieties of type A: Frobenius splitting and BW theorem}

\author{Evgeny Feigin and Michael Finkelberg}
\address{Evgeny Feigin:\newline
Department of Mathematics, National Research University Higher School of Economics,\newline
Russia, 117312, Moscow, Vavilova str. 7\newline
{\it and }\newline
Tamm Theory Division,
Lebedev Physics Institute
}
\email{evgfeig@gmail.com}
\address{Michael Finkelberg:\newline
IMU, IITP, and National Research University Higher School of Economics,\newline
Russia, 117312, Moscow, Vavilova str. 7}
\email{fnklberg@gmail.com}

\begin{abstract}
Let $\Fl^a_\la$ be the PBW degeneration of the flag varieties of type $A_{n-1}$.
These varieties are singular and are acted upon with the degenerate Lie group $SL_n^a$.
We prove that $\Fl^a_\la$ have rational singularities, are normal and
locally complete intersections, and
construct a desingularization $R_\la$ of $\Fl^a_\la$. The varieties $R_\la$ can
be viewed as  towers of successive $\bP^1$-fibrations, thus providing an analogue of the classical
Bott-Samelson-Demazure-Hansen desingularization. We prove that the varieties $R_\la$ are Frobenius
split. This gives us Frobenius splitting for the degenerate flag varieties and allows to prove the
Borel-Weil type theorem for $\Fl^a_\la$. Using the Atiyah-Bott-Lefschetz formula for
$R_\la$, we compute  the $q$-characters of the highest weight $\msl_n$-modules.
\end{abstract}

\maketitle

\section*{Introduction}
Let $\g$ be a simple Lie group and $G$ be the Lie group of $\g$. Fix a Cartan decomposition
$\g=\fb\oplus \fn^-$. Let $\g^a$ and $G^a$  be the degenerate Lie algebra and Lie group
(see \cite{Fe3}, \cite{Fe4}). Namely, $\g^a=\fb\oplus (\fn^-)^a$, where $(\fn^-)^a$
is an abelian ideal isomorphic to $\fn^-$ as a vector space and $\fb$ acts on $(\fn^-)^a$ via
the isomorphism $(\fn^-)^a\simeq \g/\fb$. The Lie group
$G^a$ is a semi-direct product of the Borel subgroup $B$ and the normal abelian
group $\bG_a^{\dim\fn}$, where $\bG_a=(\bC,+)$ is the additive group of the field.

Consider the complete flag variety $\Fl=G/B$. This variety has a degenerate version
$\Fl^a$ (see \cite{Fe3}, \cite{Fe4}). In this paper we are concerned with the
case $G=SL_n$. We denote the corresponding classical flag variety by $\Fl_n$ and
the degenerate version by $\Fl^a_n$.
For simplicity, we consider only the case of the complete flag varieties in the Introduction.
However, in the main body of the paper we work out the case of general
(parabolic) flag varieties as well.
The varieties $\Fl^a_n$ are singular projective algebraic varieties,
which can be explicitly described as follows. Fix a basis $w_1,\dots,w_n$ in an
$n$-dimensional vector space $W$ and define the projection operators $pr_d:W\to W$,
$pr_d(\sum_{i=1}^n c_iw_i)=\sum_{i\ne d} c_iw_i$. Let us denote by $Gr(d,n)$ the Grassmannian
of $d$-dimensional subspaces in $W$. Then $\Fl^a_n$ is the variety of collections
$(V_1,\dots,V_{n-1})$ of subspaces, $V_d\in Gr(d,n)$ such that
\[
pr_{d+1} V_d\subset V_{d+1},\ d=1,\dots,n-2.
\]
The group $G^a$ acts on $\Fl_n^a$ with an open $\bG_a^{\dim \fn}$-orbit.
The varieties $\Fl_a^n$ are flat degenerations of the classical flags $\Fl_n$.
Our first theorem is as follows:
\begin{thm}\label{th1}
The varieties $\Fl^a_n$ are normal locally complete intersections
(in particular, Cohen-Macaulay and even Gorenstein).
\end{thm}

Recall (see \cite{FFoL1}, \cite{FFoL2}) that for each dominant integral $\g$-weight $\la$
there exists a $\g^a$-module $V_\la^a$ which is the associated graded of $V_\la$ with respect
to the PBW filtration. Similar to the classical situation (see \cite{K}), there exists a map
$\imath_\la$ from $\Fl_n^a$ to the projectivization $\bP(V_\la^a)$
(this map is an embedding if $\la$ is regular). Therefore, one can pull back the line bundles
$\cO(1)$ from the projective space to $\Fl_n^a$. We prove the following theorem, which is the
degenerate analogue of the Borel-Weil theorem:
\begin{thm}\label{th2}
Let $\g=\msl_n$. For any dominant integral weight $\la$ one has:
\[
H^0(\Fl_n^a, \imath_\la^*\cO(1))^*\simeq V^a_\la,\ H^{>0}(\Fl_n^a, \imath_\la^*\cO(1))=0.
\]
\end{thm}

We note that this theorem agrees with the fact that the  varieties $\Fl_n^a$ are
flat degenerations of the classical flags $\Fl_n$.
Our main tool for the proof of Theorems \ref{th1} and \ref{th2} is an explicit construction for
desingularization of $\Fl_n^a$.
Namely, consider the variety $R_n$ consisting of collections of subspaces
$V_{i,j}$, $1\le i\le j\le n-1$
such that $V_{i,j}\in Gr(i,n)$ and the following conditions hold:
\begin{itemize}
\item $V_{i,j}\subset \spa(w_1,\dots,w_i,w_{j+1},\dots,w_n)$,
\item $V_{i,j}\subset V_{i+1,j},\ V_{i,j}\subset V_{i,j+1}\oplus \bC w_{j+1}$.
\end{itemize}
We show that $R_n$ is a successive tower of $\bP^1$ fibrations (and thus smooth) and the map $\pi_n:\ R_n\to \Fl_n^a$
sending $(V_{i,j})_{1\le i\le j<n}$ to $(V_{i,i})_{i=1}^{n-1}$ is a birational
isomorphism. 
Now the degenerate Borel-Weil theorem follows from the following result:
\begin{thm}\label{th3}
The varieties $\Fl_n^a$ and $R_n$ over $\overline{\mathbb F}_p$
are Frobenius split. The varieties $\Fl_n^a$ over $\overline{\mathbb F}_p$ and over $\bC$
have rational singularities.
\end{thm}
For the proof we use the Mehta-Ramanathan criterion \cite{MR}.
Using the Atiyah-Bott-Lefschetz formula (\cite{AB}, \cite{T}) we deduce from Theorem \ref{th2}
a $q$-character formula for the characters
of $V_\la^a$ (an analogue of the Demazure character formula). The formula is a sum of contributions 
of the $2^{\dim\fn}$ torus fixed points in $R_n$.

An interesting problem is to generalize the whole picture to the case of arbitrary simple Lie groups.
However the only cases worked out so far are $SL_n$ and $Sp_{2n}$ (see \cite{FFiL}). The main
obstacle comes from the complicated structure of the PBW filtration, which is not understood
outside of types $A$ and $C$.

Our paper is organized as follows:\\
In Section $1$ we recall main definitions and fix notations.\\
In Section $2$ we construct the desingularizations $R_\la$ for the degenerate flag varieties.\\
In Section $3$ we prove that the varieties $\Fl^a_\la$ are normal locally complete intersections.\\
In Section $4$ we prove that the varieties $\Fl_\la^a$ and their desingularizations are Frobenius split.\\
In Section $5$ we prove that the varieties $\Fl_\la^a$ have rational
singularities, and use the results of the previous sections to deduce the
analogue
of the Borel-Weil theorem and the $q$-character formula for $V_\la^a$.

\section{Definitions and notations}
Let $\g$ be a simple Lie algebra. Fix a Cartan decomposition $\g=\fn\oplus \h\oplus\fn^-$,
$\fb=\h\oplus\fn$.
Let $R_+$ be the set of positive roots for $\g$ and $\al_d$, $\om_d$, $d=1,\dots,\mathrm{rk}(\g)$ be the simple
roots and fundamental
weights (see \cite{FH}). For a positive root $\al$ we sometimes write
$\al>0$ instead of $\al\in R_+$. Let $f_\al$, $\al>0$ be an $\h$-eigenbasis of $\fn^-$ and, similarly,
$e_\al$ for $\fn$. We denote by $G,B,N,T,N^-$ the Lie groups of $\g,\fb,\fn,\h,\fn^-$.

Let $(\fn^-)^a$ be an abelian Lie algebra with the underlying vector space $\fn^-$.
The degenerate Lie algebra $\g^a$ is isomorphic to $\fb\oplus (\fn^-)^a$, where both
$\fb$ and $(\fn^-)^a$ are subalgebras, $(\fn^-)^a$ is an abelian ideal and the structure
of the $\fb$-module on $(\fn^-)^a\simeq\g/\fb$ is induced by the adjoint action (see \cite{Fe3}, \cite{Fe4}).
We denote
the corresponding degenerate group by $G^a$. Thus, $G^a\simeq B\ltimes (N^-)^a$, where
$(N^-)^a$ is an abelian Lie group with the Lie algebra $(\fn^-)^a$, $(N^-)^a\simeq \bG_a^M$,
where $\bG_a=(\bC,+)$ is the additive group of the field and $M=\dim\fn$ is the number of positive roots.

Let $\la$ be a dominant integral weight of $\g$ and let $V_\la$ be the corresponding irreducible $\g$-module
with a highest weight vector $v_\la$. We have $\fn v_\la=0$, $hv_\la=\la(h)v_\la$ and
$V_\la=U(\fn^-)v_\la$.
We denote by $\Fl_\la$ the generalized flag variety:
$$\Fl_\la=G\cdot \bC v_\la=\overline{N^-\cdot \bC v_\la}\subset\bP(V_\la).$$
For example, for $\g=\msl_n$ the varieties $\Fl_{\om_d}$ are isomorphic to the Grassmannians
$Gr(d,n)$ and for regular $\la$ ($(\la,\om_d)>0$ for all $d$) the corresponding flag
variety $\Fl_\la$ is isomorphic to the variety of complete flags in $\bC^n$.
Denote by $U(\fn^-)_k$ the PBW (standard) filtration of the universal
enveloping algebra
$U(\fn^-)$:
\[
U(\fn^-)_k=\spa(x_1\dots x_l, x_i\in\fn^-, l\le k).
\]
The PBW filtration $U(\fn^-)_k v_\la$ on $V_\la$ is induced by the degree filtration.
We denote by $V_\la^a$ the associated graded module:
\[
V_\la^a=\bigoplus_{k\ge 0} V_\la^a(k)=\bigoplus_{k\ge 0} U(\fn^-)_k v_\la/U(\fn^-)_{k-1} v_\la.
\]
The $q$-character of $V_\la$ (the character of $V^a_\la$)  is defined by the formula
\[
\ch_q V_\la^a=\sum_{k\ge 0} q^k \ch V_\la(k).
\]
It is easy to see that the structure of $\g$-module on $V_\la$ induces the structures of $\g^a$-
and $G^a$-module on $V_\la^a$.
In particular, $V_\la^a=\bC[f_\al]_{\al>0}v_\la$.
The corresponding degenerate flag variety $\Fl^a_\la\subset\bP(V_\la^a)$ is defined
as the closure of the orbit of the line containing $v_\la$:
\[
\Fl^a_\la=\overline{G^a \cdot \bC v_\la}=\overline{(N^-)^a \cdot \bC v_\la}.
\]
In particular, $\Fl^a_\la$ are the $\bG_a^M$-varieties (see \cite{A}, \cite{AS}, \cite{HT}).

It is convenient to consider an extension $\g^a\oplus\bC d$ of the algebra $\g^a$, where $d$
is the PBW grading operator, i.e. $[d,\fb]=0$ and $[d,f_\al]=f_\al$ for any positive $\al$.
All the $\g^a$-modules $V_\la^a$ can be made into the $\g^a\oplus\bC d$-modules by setting $d=k$ on
$V_\la^a(k)$. The corresponding extended group is $G^a\rtimes \bC^*$. In particular, the torus
acting on $\bP(V_\la^a)$ is of dimension $rk(\g)+1$.

From now on we fix $\g=\msl_n$, $G=SL_n$.
Then all positive roots are of the form $\al_{i,j}=\al_i+\dots +\al_j$, $1\le i\le j<n$.
We denote the corresponding elements $f_{\al_{i,j}}$ and $e_{\al_{i,j}}$ by $f_{i,j}$ and $e_{i,j}$.
\begin{example}
Let $\la=\om_d$. Then $V^a_{\om_d}=\bigoplus_{k=0}^{\min(d,n-d)} V^a_{\om_d}(k)$.
The space
$V^a_{\om_d}(k)$ has a basis $w(S)$ labeled by collections $S=(l_1<\dots <l_d)$ such that $1\le l_i\le n$ and
$\#\{i:\ l_i>d\}=k$. We note that $w(S)$ are the images of the wedges $w_{l_1}\wedge\dots\wedge w_{l_d}$.
The operators $f_{i,j}$ act trivially on $V^a_{\om_d}$ unless $i\le d\le j$. If this condition is satisfied, then
$f_{i,j}$ acts via the usual formula for the action on a wedge power. Similarly, the operators $e_{i,j}$ act
trivially unless $i>d$ or $j<d$. The non-trivial operators act by the usual formula.
\end{example}

In contrast with the classical situation,  a representation $V^a_{\om_d}$ is no longer isomorphic to
$\bigwedge^d(V^a_{\om_1})$. However, $V^a_{\om_d}$ can be constructed as a wedge power of another $\g^a$-module.
Namely, let $W^{(d)}$ be an $n$-dimensional vector space with a basis $w_1,\dots,w_n$. We define a structure of
$\g^a$-module on $W^{(d)}$ as follows: $f_{i,j}$ acts trivially unless $i\le d\le j$ and
$e_{i,j}$ acts trivially unless $j<d$ or $i>d$. The non-trivial operators act by the usual formulas:
\[
f_{i,j}w_k=\delta_{i,k} w_{j+1},\ e_{i,j}w_k=\delta_{j+1,k}w_i.
\]
Then $V^a_{\om_d}\simeq \bigwedge^d(W^{(d)})$.
The following simple lemma will be important for us:
\begin{lem}\label{ij}
For all $1\le i\le j<n$ the subspaces $\spa(w_{i+1},\dots,w_j)\subset W^{(i)}$
are $\g^a$-invariant, making the quotients
\[
W_{i,j}=W^{(i)}/\spa(w_{i+1},\dots,w_j)
\]
into $\g^a$- and $G^a$-modules.
\end{lem}
In what follows we denote the images in $W_{i,j}$ of the basis vectors $w_k$ by the same
symbols $w_k$. For instance, (the images of) $w_1,\dots,w_i,w_{j+1},\dots,w_n$ form a basis
of $W_{i,j}$.

\begin{example}\label{Gr}
Let $\la=\om_d$. Then $\Fl^a_{\om_d}\simeq \Fl_{\om_d}\simeq Gr(d,n)$
(since the radical in $\msl_n$ corresponding to any fundamental weight is abelian, i.e.
fundamental representations are cominuscule). The torus $T$ acts on $Gr(d,n)$ with
a finite number of fixed points, which are labeled by collections $S=(l_1,\dots,l_d)$ with
$1\le l_1<\dots <l_d\le n$. Let $p(S)\in Gr(d,n)$ be the corresponding point, i.e.
$p(S)=\bC w(S)\in \bP(V^a_{\om_d})$. Then $Gr(d,n)$ is the disjoint union of affine cells
$G^a\cdot p(S)$. We note however that these cells are different from the classical ones $B\cdot p(S)$.
Namely, let $k$ be a number such that $l_k\le d<l_{k+1}$ and let $T_d:W\to W$ be an isomorphism
given by
$$
T_d w_1=w_{d+1},\dots,T_d w_{n-d}=w_n, T_d w_{n-d+1}=w_1,\dots,T_d w_n=w_d.
$$
Then
\[
G^a\cdot p(S)=T_d\left(B\cdot p(l_{k+1}-d,\dots,l_d-d,l_1-d+n,\dots, l_k-d+n)\right),
\]
where $B$ acts on $Gr(d,n)$ classically (i.e. as a subgroup of $SL_n$). We note that $B$ considered
as a subgroup of $G^a$ acts on $Gr(d,n)$, but this action is different from the classical one
(for instance, for $n=2$ the subgroup $B\subset SL_2^a$ acts trivially on $\bP^1$).
We denote a cell $G^a\cdot p(S)$ by $C(S)$.
\end{example}

For general $\la$ the
varieties $\Fl^a_\la$ are not isomorphic to the classical flag varieties.
These varieties enjoy an explicit description as subvarieties inside the product of Grassmannians.
We first consider the case of the complete flag varieties, corresponding to the
case of regular $\la$. These varieties do not depend on (regular) $\la$. We denote them by $\Fl^a_n$.

Let $w_1,\dots,w_n$ be the standard basis of the fundamental vector representation $W=V_{\om_1}$.
We denote by $pr_d:W \to W$ the projection operators defined by
$pr_d(\sum_{i=1}^n c_iw_i)=\sum_{i\ne d} c_iw_i$.
In what follows we will need the following properties of $\Fl^a_n$ (see \cite{Fe3},\cite{Fe4}).
\begin{prop}\label{prop}
$1)$.\ The degenerate complete flag varieties $\Fl^a_n$ are flat degenerations of the classical flag varieties
$\Fl_n$.\\
$2)$.\ The variety $\Fl^a_n$ can be realized inside the product of Grassmannians
$\prod_{d=1}^{n-1} Gr(d,n)$ as a
subvariety of collections $(V_d)_{d=1}^{n-1}$ satisfying:
\[
pr_{d+1} V_d\subset V_{d+1},\ d=1,\dots,n-2.
\]
$3)$.\ The variety $\Fl^a_n$ has a cell decomposition
$$\bigsqcup_{S_1,\dots,S_{n-1}} \left(\Fl^a_n\cap \prod_{i=1}^{n-1} C(S_i)\right),$$
where the disjoint union is taken over the collections $S_1,\dots,S_{n-1}$ of subsets
$S_i\subset \{1,\dots,n\}$ such that $\# S_i=i$ and
$S_i\subset S_{i+1}\cup\{i+1\}$.
\end{prop}

There is an analogue of Proposition \ref{prop} for the degenerate partial flag varieties.
First we note that $\Fl^a_\la\simeq \Fl^a_\mu$ if and only if $(\la,\om_d)>0$ is equivalent
to $(\mu,\om_d)>0$ for any $d$. Therefore, it suffices to consider the weights
$\la=\om_{d_1}+\dots +\om_{d_k}$ with $1\le d_1<\dots <d_k<n$ and the corresponding
degenerate flag varieties $\Fl^a_\la$, which we denote by $\Fl^a_{(d_1,\dots,d_k)}$, or simply by
$\Fl^a_\bd$, where $\bd=(d_1,\dots,d_k)$. We recall
that the classical analogues $\Fl_\bd$ are isomorphic to the
partial flag varieties, i.e. to the varieties consisting of collections of subspaces $V_1,\dots,V_k$
such that $\dim V_i=d_i$ and $V_i\subset V_{i+1}$. For $1\le p\le q<n$ we define the operators $pr_{p,q}:W\to W$
via the formula
\[
pr_{p,q}(\sum_{j=1}^n c_j w_j)=\sum_{j< p} c_jw_j + \sum_{j\ge q} c_jw_j.
\]
Then the following proposition holds:
\begin{prop}\label{propgen}
$1)$.\ The degenerate partial flag varieties $\Fl_\bd^a$ are flat degenerations of the
classical partial flag  varieties $\Fl_\bd$.\\
$2)$.\ The variety $\Fl^a_\bd$ can be realized inside the product of Grassmannians
$\prod_{i=1}^k Gr(d_i,n)$ as a
subvariety of collections $(V_{d_i})_{i=1}^k$ satisfying:
\[
pr_{d_i+1,d_{i+1}} V_{d_i}\subset V_{d_{i+1}},\ i=1,\dots,k-1.
\]
$3)$.\ The variety $\Fl_\bd^a$ has a cell decomposition
$$\bigsqcup_{S_1,\dots,S_k} \left(\Fl_\bd^a\cap \prod_{i=1}^{k} C(S_i)\right),$$
where the disjoint union is taken over the collections $S_1,\dots,S_k$ of subsets
$S_i\subset \{1,\dots,n\}$ such that $\# S_i=d_i$ and
$S_i\subset S_{i+1}\cup\{d_i+1,\dots,d_{i+1}\}$.
\end{prop}

\begin{rem}
The image of the embedding $\Fl^a_\bd\subset \prod_{i=1}^k Gr(d_i,n)$ can be also described
in terms of the degenerate Pl\" ucker relations (\cite{Fe3}), similar to the classical ones
(\cite{Fu}).
\end{rem}

\section{Desingularization}\label{decomp}
\subsection{Definition}
We define a desingularization $R_n$ of the complete degenerate flag varieties $\Fl^a_n$ as follows.
Let $W_{i,j}\subset W$ be the linear span of the vectors $w_1,\dots,w_i,w_{j+1},\dots,w_n$.
\begin{dfn}\label{Rn}
The variety $R_n$ consists of collections $\bV$ of subspaces $V_{i,j}\subset W$,
$1\le i\le j\le n-1$ satisfying the following properties:
\begin{enumerate}
\item $\dim V_{i,j}=i$,\label{1}\\
\item $V_{i,j}\subset W_{i,j}$,\label{2}\\
\item $pr_{j+1} V_{i,j}\subset V_{i,j+1}\subset V_{i+1,j+1}$ for all $1\le i\le j\le n-2$.\label{3}
\end{enumerate}
\end{dfn}

\begin{rem}\label{oplus}
Since the subspace $V_{i,j}$ is embedded into $W_{i,j}$, the condition
$pr_{j+1} V_{i,j}\subset V_{i,j+1}$ is equivalent to the condition
\[
V_{i,j}\subset V_{i,j+1}\oplus\bC w_{j+1}.
\]
\end{rem}

\begin{rem}
In what follows we often identify pairs $(i,j)$ with positive roots of $\msl_n$,
$(i,j)\to \al_{i,j}$. We also sometimes consider a space $V_{i,j}$ as being attached to
the root $\al_{i,j}$ and we write $V_{\al_{i,j}}$ for $V_{i,j}$.
\end{rem}

We note that $R_n$ is naturally embedded into the product of Grassmannians
\[
R_n\hk \prod_{1\le i\le j\le n-1} Gr(i,W_{i,j}),
\]
where $Gr(i,W_{i,j})$ is the Grassmannian of $i$-dimensional subspaces in $W_{i,j}$.
Define the map $\pi_n:R_n\to\prod_{i=1}^{n-1} Gr(i,n)$ by the formula
\begin{equation}\label{pi}
\bV=(V_{i,j})_{1\le i\le j\le n-1}\mapsto (V_{1,1},\dots,V_{n-1,n-1}).
\end{equation}

\begin{prop}\label{birat}
The image of $\pi_n$ is equal to $\Fl^a_n$.
The variety $R_n$ is smooth and the map $\pi_n:R_n\to \Fl^a_n$ is a birational isomorphism.
\end{prop}
\begin{proof}
We note that if $\bV\in R_n$ then
\[
pr_{i+1} V_{i,i}\subset V_{i,i+1}\subset V_{i+1,i+1}
\]
and thus $\pi_n(R_n)\subset \Fl^a_n$. Now given an element $(V_1,\dots,V_{n-1})\in\Fl_n^a$,
we define a collection $\bV$ via the following inductive procedure: $V_{i,i}=V_i$ and
\[
V_{i,j+1}=
\begin{cases}
pr_{j+1} V_{i,j},\ \text{ if } \dim pr_{j+1} V_{i,j}=i,\\
pr_{j+1} V_{i,j}\oplus \bC w_m,\ \text{ if } \dim pr_{j+1} V_{i,j}=i-1,
\end{cases}
\]
where $m\in\{1,\dots,i\}$ is the minimal number such that $w_m\notin pr_{j+1} V_{i,j}$.
Then it is easy to see that $\bV=(V_{i,j})$ belongs to $R_n$. Hence $\pi_n$ surjects $R_n$
to $\Fl_n^a$.

Now we show that $R_n$ can be viewed as a tower of $\bP^1$-fibrations. Let us order
all positive roots of $\msl_n$ as follows:
\[
\beta_1=\al_{1,n-1}, \beta_2=\al_{1,n-2}, \beta_3=\al_{2,n-1}, \beta_4=\al_{1,n-3}, \beta_5=\al_{2,n-2},\dots.
\]
Let $R_n(k)$, $k=1,\dots,n(n-1)/2$ be the variety of collections $(V_{\beta_l})_{l=1,\dots,k}$, satisfying
properties \eqref{1}, \eqref{2}, \eqref{3} from Definition \ref{Rn}
(conditions \eqref{1}, \eqref{2} and \eqref{3} are applied only to those $V_{i,j}$ which show up in $R_n(k)$,
i.e. for $\beta_l=\al_{i,j}$ one has $l\le k$).
Then $R_n(n(n-1)/2)=R_n$ and there exist obvious projections $R_n(k)\to R_n(k-1)$.
We prove that for all $k\ge 1$  the projections $R_n(k)\to R_n(k-1)$ are fibrations
with  fibers $\bP^1$ (we set $R_n(0)=pt$).

For $k=1$ we have $\beta_1=\al_{1,n-1}$ and $V_{\beta_1}$ is a one-dimensional space embedded into
two-dimensional space $\mathrm{span}(w_1,w_{n-1})$. Therefore $R_n(1)\simeq \bP^1$. Now fix some $k$.
Let $\beta_k=\al_{i,j}$. First, let $i\ne 1$, $j\ne n-1$. Then the $i$-dimensional subspace $V_{i,j}$
has to satisfy the conditions
\begin{equation}\label{fibr}
V_{i-1,j}\subset V_{i,j}\subset V_{i,j+1}\oplus \bC w_{j+1}
\end{equation}
(see Remark \ref{oplus}). Suppose we have fixed all subspaces $V_{\beta_l}$ with $l<k$
(i.e. a point of $R_n(k-1)$). Since
$V_{i-1,j}$ and $V_{i,j+1}$ are already fixed, conditions \eqref{fibr} say that the possible choices of
$V_{i,j}$ are labeled by points of
\[
\bP^1\simeq \bP\left(\frac{V_{i,j+1}\oplus \bC w_{j+1}}{V_{i-1,j}}\right).
\]
Second, let $i=1$. Then we have to fix a one-dimensional subspace $V_{1,j}$ living in a two-dimensional space
$V_{1,j+1}\oplus\bC w_{j+1}$. This gives us again a $\bP^1$-fibration.
Finally, let $j=n-1$. Then we need to fix an $i$-dimensional subspace $V_{i,n-1}$ subject to the
conditions
\[
V_{i-1,n-1}\subset V_{i,n-1}\subset \spa(w_1,\dots,w_i,w_n),
\]
which again produces $\bP^1$.

It remains to prove that the map $\pi_n:R_n\to\Fl^a_n$ is a birational isomorphism.
Consider the subvariety $U\subset \Fl^a_n$ consisting of all collections of subspaces $(V_i)_{i=1}^{n-1}$
such that $\dim V_i=i$ and
\[
\dim pr_{i+1}\dots pr_{n-1} V_i=i
\]
(i.e. the composition of the projections as above has no kernel on $V_i$).
First note that these conditions cut out an open subvariety in $\Fl^a_n$ (in fact it is easy to see that
$U$ is an affine cell). In addition, the preimage $\pi_n^{-1} (V_i)_{i=1}^{n-1}$ consists of a single point,
since
\[
V_{i,j}\subset pr_{j}\dots pr_{i+1} V_{i,i}
\]
and both spaces are $i$-dimensional.
\end{proof}

\begin{rem}
As we have seen in the proof of Proposition \ref{birat}, the variety $R_n$ can be constructed as a tower of
successive $\bP^1$-fibrations $\rho_k:R_n(k)\to R_n(k-1)$. We can make this statement a bit stronger.
Let us write $\rho_k=\rho_{i,j}$ if $\beta_k=\al_{i,j}$.
Then it is easy to see that the maps
$\rho_{i,j}$ with fixed $j-i$ "commute", i.e. for each $m=n-1,\dots,1$
there exist maps
$$\bar\rho_m: R_n(m(m+1)/2)\to R_n((m-1)m/2) ,$$ which are the $(\bP^1)^m$ fibrations, and
$\bar\rho_m=\prod_{i=1}^m\rho_{i,i+n-m-1}$.
\end{rem}

Denote by $\xi_{i,j}:R_n\to Gr(i,W_{i,j})$ the projection given by $\bV\mapsto V_{i,j}$.
\begin{lem}\label{ind}
For any $k=0,\dots,n-2$ the image of the map $\prod_{j-i=k} \xi_{i,j}$ is isomorphic to $\Fl^a_{n-k}$.
\end{lem}
\begin{proof}
Recall that $V_{i,j}\subset W_{i,j}$. Consider an isomorphism
$$A_{i,j}:W_{i,j}\to\spa(w_1,\dots,w_{n-j+i})$$ defined by
\begin{multline*}
A_{i,j}(c_1w_1+\dots +c_iw_i+c_{j+1}w_{j+1}+\dots +c_nw_n)=\\
c_1w_1+\dots+c_iw_i+c_{j+1}w_{i+1}+\dots + c_nw_{n-j+i}.
\end{multline*}
Then it is easy to see that this map induces the isomorphism stated in our lemma.
\end{proof}

\begin{cor}
We have an embedding $R_n\hk\prod_{k=1}^{n-1} \Fl^a_k$.
\end{cor}

\begin{cor}\label{indcor}
The varieties $R_n(k(k+1)/2)$ and $R_k$ are isomorphic.
\end{cor}
\begin{proof}
We note that the spaces involved in the construction of $R_n(k(k+1)/2)$ are exactly $V_{i,j}$ with
$j-i\ge n-k-1$. Now the maps $A_{i,j}$ as above induce the desired isomorphism.
\end{proof}

We now consider the case of partial flag varieties $\Fl^a_\bd$ (recall the notation
$\bd=(d_1,\dots,d_k)$).
Let $P_{(d_1,\dots,d_k)}=P_\bd$ be the subset of the set of positive roots of $\msl_n$
corresponding to the radical of the parabolic
subalgebra defined by the simple roots $\al_{d_1},\dots, \al_{d_k}$, i.e.
\[
P_\bd=\{\al_{i,j}:\ \exists\ l \text{ such that } (\al_{i,j},\om_{d_l})>0\}.
\]
We sometimes consider $P_\bd$ as a subset of $\bN^2$ identifying $\al_{i,j}$ with the
pair $(i,j)$.
Let $R_\bd$ be the
image of the map
$$\prod_{(i,j)\in P_\bd}\xi_{i,j}: R_n\to \prod_{(i,j)\in P_\bd} Gr(i,W_{i,j}).$$
More concretely, $R_\bd$ is the
variety of collections $V_{i,j}\subset W$ with $(i,j)\in P_\bd$
satisfying conditions
\begin{enumerate}
\item $\dim V_{i,j}=i$,\\
\item $V_{i,j}\subset W_{i,j}$,\\
\item $pr_{j+1} V_{i,j}\subset V_{i,j+1}$ if  $(i,j),(i,j+1)\in P_\bd$,\\
\item $V_{i,j+1}\subset V_{i+1,j+1}$ if  $(i,j+1),(i+1,j+1)\in P_\bd$
\end{enumerate}
from Definition \ref{Rn}. Obviously, $R_n$ surjects to $R_\bd$ by forgetting all components $V_{i,j}$
but those with $(i,j)\in P_\bd$.
\begin{prop}
For any $\bd$ the variety $R_\bd$ is smooth and a natural map $R_\bd\to \Fl^a_\bd$ defined by
forgetting the off-diagonal
($i\ne j$) subspaces $V_{i,j}$ is a desingularization (a birational isomorphism).
\end{prop}
\begin{proof}
The proof is very similar to the proof for the complete flag varieties.
\end{proof}

\begin{rem}
In the Introduction the varieties $R_\bd$ are denoted by $R_\la$ with
$\lambda=\omega_{d_1}+\ldots+\omega_{d_k}$.
\end{rem}

\subsection{Cell decomposition for $R_n$}
In this section we construct a cell decomposition for $R_n$ which is compatible with the
cell decomposition for
$\Fl^a_n$ (i.e. the map $\pi_n$ is cellular).

\begin{lem}
The group $G^a$ acts naturally on each $Gr(i,W_{i,j})$. The number of $G^a$-orbits is
finite and  the orbits are labeled by torus fixed points. Each orbit is an affine cell.
\end{lem}
\begin{proof}
Fix a pair $1\le i\le j\le n-1$. Recall that $V_{i,j}\subset W_{i,j}$. Therefore,
$V_{i,j}$ can be considered as a point in $\bP(\bigwedge^i(W_{i,j}))$.
The spaces $\bigwedge^i(W_{i,j})$ carry a natural structure
of $\g^a$- and $G^a$-modules (see Lemma \ref{ij}). This produces a $G^a$-action on $\bP(\bigwedge^i(W_{i,j}))$
and thus on the variety $Gr(i,W_{i,j})$ of $i$-dimensional subspaces of $W_{i,j}$.

Let us consider the smaller group $SL_{n-j+i}^a$. Using the maps $A_{i,j}$ from Lemma \ref{ind}
we endow $W_{i,j}$ and all its wedge powers with the standard structure of $SL_{n-j+i}^a$-modules
(we identify $\bigwedge^k(W_{i,j})$ with the $SL_{n-j+i}^a$-modules $V_{\om_k}^a$).
Let $\phi: SL_{n-j+i}^a\to GL(\bigwedge^k(W_{i,j}))$ be the representation map and also let
$\psi: SL_n^a\to GL(\bigwedge^k(W_{i,j}))$ be the map defining the $G^a$ action on $\bigwedge^k(W_{i,j})$.
It is easy to see that the images of $\phi$ and $\psi$ coincide. Therefore,
Example \ref{Gr} (applied to the group $SL_{n-j+i}^a$) implies the statement of our Lemma.
\end{proof}

We note that the torus fixed points in the Grassmannian of $i$-dimensional subspaces in $W_{i,j}$
are labeled by the sequences
$$S=(l_1< \dots <l_i)\subset \{1,\dots,i,j+1,\dots,n\}.$$
In what follows we denote the corresponding point by $p(S)$.
We also denote the corresponding orbit $G^a\cdot p(S)\subset Gr(i,W_{i,j})$ by $C(S)$.

Recall that $R_n$ sits inside $\prod_{1\le i\le j\le n-1} Gr(i,W_{i,j})$. The group $G^a$ acts on
this product via the action on each factor.
\begin{lem}
The variety $R_n$ is invariant with respect to this action and $\pi_n:R_n\to \Fl^a_n$ is $G^a$-equivariant.
\end{lem}
\begin{proof}
First, take $b\in B\subset G^a$ and fix a point $\bV=(V_{i,j})\in R_n$. We need to show that
for any $1\le i\le j\le n-1$
\begin{equation}\label{b}
bV_{i-1,j}\subset bV_{i,j}\subset bV_{i,j+1}\oplus \bC w_{j+1}.
\end{equation}
We note that $W_{i-1,j}\subset W_{i,j}$ and the $B$-action on $W_{i-1,j}$ is
a restriction of the action on $W_{i,j}$. Therefore, the first embedding in \eqref{b} follows.
To prove the second embedding we note that $\bC w_{j+1}$ is a $\fb$-submodule in $W_{i,j}$ and
the quotient module is isomorphic to $W_{i,j+1}$.

Now take $g\in (N^-)^a$. We need to prove that
for any $1\le i\le j\le n-1$
\[
gV_{i-1,j}\subset gV_{i,j}\subset gV_{i,j+1}\oplus \bC w_{j+1}.
\]
The proof is very similar to the proof of \eqref{b} and we omit it.
\end{proof}

Let ${\bf S}=(S_{i,j})_{1\le i\le j\le n-1}$ be a collection of sets such that $\# S_{i,j}=i$ and
$S_{i,j}\subset \{1,\dots,i,j+1,\dots,n\}$. We call such a collection admissible if
\begin{equation}\label{admis}
S_{i-1,j}\subset S_{i,j}\subset S_{i,j+1}\cup \{j+1\}.
\end{equation}
The following lemma is simple, but important for us.
\begin{lem}
A point $p(\bS)=\prod_{1\le i\le j\le n-1} p(S_{i,j})$ belongs to $R_n$
if and only if $\bS$ is admissible. If a point $p=\prod_{1\le i\le j< n} p_{i,j}$, $p_{i,j}\in C(S_{i,j})$
belongs to $R_n$, then the collection $\bS=(S_{i,j})$ is admissible.\qed
\end{lem}

For an admissible collection $\bS$ we introduce the notation
\[
C(\bS)=R_n\cap \prod_{1\le i\le j\le n-1} C(S_{i,j}).
\]
We have the decomposition
\[
R_n=\bigcup_{\text {admissible } \bS} C(\bS).
\]
Our next goal is to show that $C(\bS)$ is an affine cell and
to compute it dimension.

For a number $l$, $-n<l\le n$ we set $\overline{l}=l$ if $l>0$ and
$\overline{l}=l+n$ otherwise. So $1\le \overline{l}\le n$.

\begin{thm}\label{cells}
$C({\bS})$ is an affine cell for any admissible ${\bf S}$. The map $\pi_n$ is cellular,
mapping $C(\bS)$ to $C(S_{1,1},\dots,S_{n-1,n-1})$.
\end{thm}
\begin{proof}
We need to do two things: first, to construct coordinates on a cell $C(\bS)$ and, second,
to construct coordinates on each fiber of the map $$C(\bS)\to C(S_{1,1},\dots,S_{n-1,n-1})$$ such that
this map becomes a trivial fibration with an affine fiber. We start with the first part.

We want to construct coordinates on $C({\bf S})$. Namely, we need to attach coordinates to collections
of subspaces $(V_{i,j})_{1\le i\le j< n}\in R_n$. We do it by decreasing induction on $j-i$. We start
with $j-i=n-2$, i.e. $i=1$, $j=n-1$. Then either $S_{1,n}=(n)$ or $S_{1,n}=(1)$. In the first case
the cell $C((n))$ is a point and in the second case $V_{1,n-1}$ is spanned by a single vector
$v_1+av_n$ and $a$ is our first coordinate.
Assume that we have attached coordinates to all subspaces $V_{i,j}$ with $j-i>k$ and we proceed with $j-i=k$.
We consider three cases.

Let $i=1$. Then the only condition we have is $V_{1,j}\subset V_{1,j+1}\oplus\bC v_{j+1}$. Let $S_{1,j}=(l)$.
There are two cases: $l=j+1$ and $l\ne j+1$. In the first case
we do not have to add any coordinates, since $C((j+1))\subset Gr(1,W_{1,j})$  is a point.
Let $l\ne j+1$ and let
$v\in V_{1,j+1}$ be a basis vector. Then a basis vector for $V_{1,j}$ is of the form $v+aw_{j+1}$ and therefore we
have added one more coordinate.

Let $j=n-1$. Then we have the condition $V_{i-1,n-1}\subset V_{i,n-1}$. We know that
$S_{i,n-1}=S_{i-1,n-1}\cup \{l\}$. There are two cases: $l=i$ and $l\ne i$. First, let $l=i$.
Let $m=\{1,\dots,i-1,n\}\setminus S_{i-1,n}$. Since $V_{i-1,n}$ is $(i-1)$-dimensional, we need to
specify one more basis vector in $V_{i,n-1}$ in order to fix it. This basis vector has to be of the form
\[
w_i+c_{i-1}w_{i-1}+\dots + c_1w_1+c_n w_n,\ c_k\in\bC
\]
We note that by adding an appropriate vector from $V_{i-1,n-1}$, any vector of the form as above can be
reduced to $w_i+aw_m$. This gives one additional coordinate.
Second, let $l\ne i$. Then $l=\{1,\dots,i-1,n\}\setminus S_{i-1,n}$. A basis vector we have to add to $V_{i-1,n-1}$
in order to fix $V_{i,n-1}$ is of the form
\[
w_l+c_{l-1}w_{l-1}+\dots + c_1w_1+c_n w_n,\ c_k\in\bC.
\]
Since $w_i$ never appears in the decomposition as above, such a vector
(modulo $V_{i-1,n-1}$) is equal to $w_l$ and
we do not have to add a coordinate.

Let $i>1$, $j<n-1$. Then we have
\[
S_{i-1,j}\subset S_{i,j}\subset S_{i,j+1}\cup\{j+1\}, \
V_{i-1,j}\subset V_{i,j}\subset V_{i,j+1}\oplus\bC w_{j+1}.
\]
First, let $S_{i,j}=S_{i,j+1}$, i.e. $j+1\notin S_{i,j}$. Let $l=S_{i,j}\setminus S_{i-1,j}$.
Then a basis  vector we have
to add to $V_{i-1,j}$ in order to fix $V_{i,j}$ is of the form
\[
w_l+c_{l-1}w_{l-1}+\dots+ c_1w_1+c_n w_n+\dots +c_{j+1}w_{j+1}.
\]
Since this vector has to belong to $V_{i,j+1}$, the only freedom we have is a coefficient $c_{j+1}$
(note that $l\ne j+1$).
Therefore, we have to add one additional coordinate in this  case.
Second, let $S_{i,j}\ne S_{i,j+1}$, i.e. $j+1\in S_{i,j}$. Then $S_{i,j}=S_{i,j+1}\setminus\{m\}\cup\{j+1\}$.
Recall $l=S_{i,j}\setminus S_{i-1,j}$. A basis  vector we have
to add to $V_{i-1,j}$ in order to fix $V_{i,j}$ is of the form
\begin{equation}\label{form}
w_l+c_{l-1}w_{l-1}+\dots+ c_1w_1+c_n w_n+\dots +c_{j+1}w_{j+1}.
\end{equation}
Recall that for a number $l$, $-n<k\le n$ we set $\overline{l}=l$ if $l>0$ and
$\overline{l}=l+n$ otherwise.
There are two cases now: $\ol{l-j}<\ol{m-j}$ and $\ol{l-j}>\ol{m-j}$. Let $\ol{l-j}<\ol{m-j}$.
Then the vector $w_m$ never appears
in the decomposition \eqref{form} and therefore there exists a single vector in $V_{i,j+1}$ of the form  \eqref{form}.
Thus no new coordinates have to be added. Finally, let $\ol{l-j}>\ol{m-j}$. Then
a vector $w_m$ is present in \eqref{form}.
Therefore, there exists exactly one-parameter family of vectors in $V_{i,j+1}$ of the form \eqref{form}.
Thus one additional coordinate has to be added.

To complete the proof of the theorem we need to construct coordinates on the fibers of the map
$C(\bS)\to C(S_{1,1},\dots,S_{n-1,n-1})$. To do this, one need to fix a collection of subspaces $V_{i,i}\in C(S_{i,i})$
such that $(V_{i,i})_{i=1}^{n-1}\in\Fl^a_n$ and then start looking at all possible values of other
$V_{i,j}\in C(S_{i,j})$ moving from lower values of $j-i$ to higher ones. The procedure is very
similar to the one worked out above, so we omit the details.
\end{proof}

\begin{cor}\label{dc}
For an admissible $\bS$ the dimension of the cell $C({\bf S})$
is equal to the sum of $n(n-1)/2$ terms $g_{i,j}$ labeled by pairs $1\le i\le j\le n-1$.
Each summand is either $0$ or $1$ and is given by the following rule:
\begin{itemize}
\item Let $i=1$, $j=n-1$. If $S_{1,n}=(1)$, then $g_{1,n-1}=1$. Otherwise $g_{1,n-1}=0$.
\item Let $i=1$ and $S_{1,j}=(l)$. If $l\ne j+1$, then $g_{i,j}=1$. Otherwise $g_{i,j}=0$.
\item Let $j=n-1$. Let $\{l\}=S_{i,n-1}\setminus S_{i-1,n-1}$. If $l=i$, then $g_{i,j}=1$.
Otherwise $g_{i,j}=0$.
\item Let $i>1$ and $j<n-1$.
\subitem If $j+1\notin S_{i,j}$, then $g_{i,j}=1$.
\subitem  If $j+1\in S_{i,j}$, set $l=S_{i,j}\setminus S_{i-1,j}$, $m=S_{i,j+1}\setminus S_{i,j}$.
If $\ol{l-j}>\ol{m-j}$, then $g_{i,j}=1$. Otherwise $g_{i,j}=0$.
\end{itemize}
\end{cor}
\begin{proof}
Follows from the explicit construction of the coordinates on $C(\bS)$.
\end{proof}

\begin{cor}\label{cdc}
The relative dimension $\dim C({\bf S})-\dim C(S_{i,i})_{i=1}^{n-1}$
is equal to the sum of $(n-1)(n-2)/2$ terms $h_{i,j}$ labeled by pairs $1\le i< j\le n-1$.
Each summand is either $0$ or $1$ and is given by the following rule.
Let $l=S_{i,j}\setminus S_{i,j-1}$, $m=S_{i+1,j}\setminus S_{i,j}$. Then $h_{i,j}=0$ if
and only if  $\ol{m-j}<\ol{l-j}$ and $j\in S_{i,j-1}$.
\end{cor}
\begin{proof}
Follows from the explicit construction of the coordinates on $C(\bS)$.
\end{proof}

We note that the desingularization $\pi_n$ is small up to $n=4$, semismall 
up to $n=7$, but not semismall starting from $n=8$.   


Finally, we note that Theorem \ref{cells} as well as Corollaries \ref{dc} and \ref{cdc} have their
obvious parabolic analogues. Namely, let us call a collection
$\bS=(S_{i,j})_{(i,j)\in P_\bd}$ $\bd$-admissible,
if condition \eqref{admis} holds provided the corresponding pairs of indices are in $P_\bd$.
Then the following theorem holds:
\begin{prop}
$1)$.\ $R_\bd$ is a disjoint union of the cells
$$\bigsqcup_{\bd-\text{admissible }\bS} \left(R_\bd\cap \prod_{(i,j)\in P_\bd} C(S_{i,j})\right).$$
$2)$.\ The map $R_\bd\to \Fl^a_\bd$ is cellular.\\
$3)$.\ The dimensions and relative dimensions are equal to the sum of terms $g_{i,j}$ and $h_{i,j}$ from Corollaries \ref{dc}
and \ref{cdc} with $(i,j)\in P_\bd$.
\end{prop}

\section{Normality}
\subsection{Complete flag varieties.}
We first construct a quiver realization of the complete degenerate flag
varieties. Let $W_1, \dots, W_{n-1}, W_n$ be a collection of fixed spaces with $\dim W_i=i$.
Additionally, we fix a basis $e_1,\dots,e_n$ in $W_n$ and the projections $pr_k$ along $e_k$.
We now construct an affine scheme $Q_n$ as follows. A point of $Q_n$ is
a collection of linear maps
$$A_i: W_i\to W_n,\ i=1,\dots,n-1,\qquad  B_j: W_j\to W_{j+1},\ j=1,\dots,n-2$$
subject to the relations
\begin{equation}\label{cd}
A_{i+1}B_i=pr_{i+1}A_i,\ i=1,\dots,n-2.
\end{equation}
The following picture illustrates  the construction:
$$
\begin{picture}(250,80)
\put(10,0){$W_1$}
\put(30,3){\vector(1,0){30}}
\put(70,0){$W_2$}
\multiput(100,0)(5,0){10}{\circle*{1}}
\put(160,0){$W_{n-2}$}
\put(190,3){\vector(1,0){45}}
\put(240,0){$W_{n-1}$}

\put(10,50){$W_n$}
\put(70,50){$W_n$}
\put(160,50){$W_n$}
\put(240,50){$W_n$}
\put(30,53){\vector(1,0){30}}
\put(185,53){\vector(1,0){50}}
\multiput(100,50)(5,0){10}{\circle*{1}}

\put(0,24){$A_1$}
\put(75,24){$A_2$}
\put(166,24){$A_{n-2}$}
\put(245,24){$A_{n-1}$}
\put(40,6){$B_1$}
\put(200,6){$B_{n-2}$}

\put(15,14){\vector(0,1){31}}
\put(75,14){\vector(0,1){31}}
\put(165,14){\vector(0,1){31}}
\put(245,14){\vector(0,1){31}}

\put(38,57){$pr_2$}
\put(198,57){$pr_{n-1}$}
\end{picture}
$$

We also consider an open part $Q_n^\circ\subset Q_n$ consisting of collections $(A_i, B_j)$
such that $\ker A_i=0$ for all $i$.  
The group $\Gamma=\prod_{i=1}^{n-1} GL(W_i)$ acts freely on $Q^\circ_n$
via the change of bases. Consider the map 
\[
Q^\circ_{n}\to \Fl^a_{n},\quad (A_i,B_j)\mapsto (\mathrm{Im} A_1,\dots ,\mathrm{Im} A_{n-1}). 
\]

\begin{lem}
The map $Q^\circ_{n}\to \Fl^a_{n}$ is locally trivial $\Gamma$-torsor in the Zariski topology.
The dimension of $Q^\circ_n$ (and thus of $Q_n$) is equal to $n(n-1)/2 + 1^2+2^2+\ldots+(n-1)^2$.
\end{lem}
\begin{proof}
Consider the embedding $\Fl^a_{n}\hk \prod_{d=1}^{n-1} Gr(d,n)$. 
For a point $p\in \Fl^a_{n}$ let $U\ni p$ be an open part of $\prod_{d=1}^{n-1} Gr(d,n)$
such that all tautological bundles on Grassmannians are trivial on $U$. Let 
$U'=U\cap \Fl^a_{n}$. Then on $U'$ the map $Q^\circ_{n}\to \Fl^a_{n}$ has a section.
Now using the $\Gamma$ action on $Q_{n}$ we obtain that $Q^\circ_n\to \Fl^a_{n}$ is 
$\Gamma$-torsor. In particular, $\dim Q^\circ_n=\dim Q_n=\dim \Fl^a_n + \dim \Gamma$.
\end{proof}

We note that $Q_n$ is a subscheme in the affine
space
\begin{equation}\label{affine}
\prod_{i=1}^{n-1} Hom(W_i,W_n)\times \prod_{i=1}^{n-2} Hom(W_i,W_{i+1}).
\end{equation}

\begin{lem}
$Q_n$ is a complete intersection.
\end{lem}
\begin{proof}
The condition $A_{i+1}B_i=pr_{i+1}A_i$ produces $n\times i$ equations
(the number of equations is equal to $\dim Hom(W_i,W_n)$). Now our lemma follows from the equality
$$
\dim Q_n= \sum_{i=1}^{n-1} ni + \sum_{i=1}^{n-2} i(i+1) -\sum_{i=1}^{n-2} ni.
$$
\end{proof}

\begin{thm}
The degenerate flag varieties $\Fl^a_n$ are normal locally complete
intersections (in particular, Cohen-Macaulay and even Gorenstein).
\end{thm}
\begin{proof}
Since $Q^\circ_n\to \Fl^a_n$ is a torsor, it suffices to prove that 
$Q^\circ_n$ is a normal reduced scheme (i.e. a variety).
Since $Q^\circ_n$ is locally complete intersection, the property of being reduced
(resp. normality) of $Q^\circ_n$ follows from the fact
that the singularities of $Q^\circ_n$ are contained in the subvariety of codimension
at least two by the virtue of~Proposition~5.8.5 (resp.~Theorem~5.8.6)
of~\cite{ega}. Using again that $Q^\circ_n\to \Fl^a_n$ is a torsor,
it suffices to prove that
the codimension of the variety of singular points of $\Fl^a_n$ is at least two.
We prove this statement in a separate lemma.
\end{proof}

\begin{lem}\label{smooth}
$\Fl^a_n$ is smooth off codimension two.
\end{lem}
\begin{proof}
There are two ways to prove the statement. The first one uses the representation theory of quivers and
is worked out in \cite{CFR}, Theorem 4.5. The second way is more direct. Namely, let us
use the desingularization $\pi_n:R_n\to \Fl^a_n$. Since $R_n$ is smooth, it
suffices to show that that the map $\pi_n$ is an isomorphism on all cells of (complex) codimension one.
Dimension counting from Corollary \ref{dc} implies that the codimension one cells are labeled by
pairs $1\le a\le b\le n-1$ and the collection $\bS=(S_{i,j})$ corresponding to a pair $(a,b)$  is as follows:
\begin{equation}\label{codim1}
S_{i,j}=
\begin{cases}
\{1,2,\dots,i\} \text{ if } (i<a \text{ or } j>b),\\
\{1,2,\dots,i\}\setminus \{ a\} \cup \{b+1\}, \text{otherwise.}
\end{cases}
\end{equation}
It is easy to see from Corollary \ref{cdc} that the resolution map $\pi_n$
is an isomorphism on such cells.
\end{proof}


\subsection{Parabolic flag varieties.}
Our goal is to generalize the results from the previous subsection to the case of the general parabolic
degenerate flag varieties. So let $\bd=(d_1,\dots,d_k)$ be a collection with $1\le d_1<\dots <d_k\le n$.
We define an affine scheme $Q_\bd$ as follows. As above, we fix the spaces $W_{d_i}$, $i=1,\dots,k$ 
with $\dim W_{d_i}=d_i$. A point of $Q_\bd$ is
a collection of linear maps
$$A_i: W_{d_i}\to W_n,\ i=1,\dots,k,\qquad  B_j: W_{d_j}\to W_{d_{j+1}},\ j=1,\dots,k-1$$
subject to the relations
\begin{equation}\label{cd1}
A_{i+1}B_i=pr_{d_i+1}\dots pr_{d_{i+1}}A_i,\ i=1,\dots,n-2.
\end{equation}
We also define $Q_\bd^\circ\subset Q_\bd$ as an open part defined by the conditions $\ker A_i=0$ for
all $i$. 
The group $\Gamma_\bd=\prod_{i=1}^{k} GL(W_{d_i})$ acts freely on $Q^\circ_\bd$
via the change of bases and, as in the case of the complete flag varieties,  
$Q^\circ_\bd/\Gamma_\bd\simeq \Fl^a_\bd$, i.e. $Q^\circ_\bd$ is a
$\Gamma_\bd$-torsor over $\Fl^a_\bd$.
Moreover,  explicit computation as above shows that
$Q_\bd$  is a complete intersection.
Now the following theorem holds:
\begin{thm}
The degenerate flag varieties $\Fl^a_\bd$ are normal locally complete
intersections (in particular, Cohen-Macaulay and even Gorenstein).
\end{thm}

Again, as in the complete case, we only need to prove that each variety
$\Fl^a_\bd$ is smooth outside of the codimension two subvariety. This is proved for a wider class
of varieties in \cite{CFR}. Here we present a more direct proof. 

\begin{prop}\label{Yd}
$\Fl^a_\bd$ is smooth off codimension two.
\end{prop}
\begin{proof}
As in Lemma \ref{smooth} it suffices to construct a desingularization $Y_\bd$ of $\Fl^a_\bd$
such that the map $\tau_\bd:Y_\bd\to \Fl^a_\bd$ is one-to-one off codimension two. Unfortunately, $R_\bd$ does not do the
job (it is too big). We refine it in the following way. Let $Y_\bd$ be the variety of subspaces
$V_{d_i,d_j}$, $1\le i\le j\le k$ satisfying the following properties:
\begin{gather*}
\dim V_{d_i,d_j}=d_i,\quad V_{d_i,d_j}\subset W_{d_i,d_j},\\
V_{d_i,d_j}\subset V_{d_{i+1},d_j},\quad
pr_{d_j+1}\dots pr_{d_{j+1}} V_{d_i,d_j}\subset V_{d_i,d_{j+1}}.
\end{gather*}
The projection map $\tau_\bd$ is defined by $(V_{d_i,d_j})_{i,j=1}^k\mapsto (V_{d_i,d_i})_{i=1}^k$
(i.e. simply forgetting the off-diagonal entries).

The varieties $Y_\bd$ are smooth and can be
viewed as towers of fibrations with fibers isomorphic to the Grassmann
varieties. More precisely,
these towers are constructed as follows. First, the subspace
$V_{d_1,d_k}$ varies in $Gr(d_1,W_{d_1,d_k})$.
Second, we consider $V_{d_1,d_{k-1}}$ and $V_{d_2,d_k}$. For the former,
the only condition is
\[
V_{d_1,d_{k-1}}\subset V_{d_1,d_{k}}\oplus\spa(w_{d_{k-1}+1},\dots,w_{d_k}),
\]
which produces the fibration over $Gr(d_1,W_{d_1,d_k})$ with a fiber
$Gr(d_1,d_1+d_k-d_{k-1})$. Now the conditions for
$V_{d_2,d_k}$ are $V_{d_1,d_k}\subset V_{d_2,d_k}\subset W_{d_2,d_k}$,
producing a fibration over $Gr(d_1,W_{d_1,d_k})$ with a fiber
$Gr(d_2-d_1, n-d_k+d_2-d_1)$. Proceeding further, we see that
$Y_\bd$ is a tower of fibrations with fibers being Grassmannians.

As in the case of complete flag varieties, the varieties $Y_\bd$ possess a
cellular decomposition. Namely,
the cells are labeled by collections $\bS=(S_{d_i,d_j})$, $1\le i\le j\le k$ satisfying the usual properties
\begin{gather*}
\# S_{d_i,d_j}=d_i,\quad S_{d_i,d_j}\subset \{1,\dots,d_i,d_{j}+1,n\},\\
S_{d_i,d_j}\subset S_{d_{i+1},d_j}\subset S_{d_{i+1},d_{j+1}}\cup\{d_j+1,\dots,d_{j+1}\}.
\end{gather*}
A cell $C(\bS)$ is defined as the intersection $Y_\bd\cap \prod_{i,j}
C(S_{d_i,d_j})$.
For example, the big cell in $Y_\bd$ is given by
$S_{d_i,d_j}=\{1,\dots,d_i\}$. It is easy to see that $\tau_\bd$ is one-to-one on this cell.
In order to prove the proposition it suffices to show that $\tau_\bd$ is an isomorphism
on all cells of codimension one. Let us describe these cells.

First consider a single Grassmannian $Gr(d,n)$. The unique codimension one cell is
$C(S)$ with $S=\{2,\dots,d,n\}$. Using this observation and the construction of $Y_\bd$ as a tower
of successive fibrations with fibers being Grassmanians, we obtain the following description of
codimension one cells in $Y_\bd$. These cells are labeled by pairs $1\le a\le b\le k$ and
a collection $\bS$ corresponding to such a pair is given by
$$
S_{d_i,d_j}=
\begin{cases}
\{1,2,\dots,d_i\}, \text{  if  } (i<a \text{ or } j>b),\\
\{1,2,\dots,d_i\}\setminus \{d_{a-1}+1\} \cup \{d_{b+1}\}, \text{  otherwise}
\end{cases}
$$
(compare with \eqref{codim1}). It is easy to check that the map $\tau_\bd$ is an isomorphism on such cells.
\end{proof}

\section{Frobenius splitting}
The goal of this section is to show that the varieties $\Fl_n^a$ over $\overline{\mathbb F}_p$ are Frobenius
split. The general references are~\cite{MR},~\cite{BK}.
We first recall the definition. Let  $X$ be an algebraic variety over an algebraically
closed field of characteristic $p>0$.
Let $F:X\to X$ be the Frobenius morphism, i.e. the identity map on the underlying space
$X$ and the $p$-th power map on the space of functions.
Then $X$ is called Frobenius split if there exists a projection $F_*\cO_X\to \cO_X$ such that
the composition $\cO_X\to  F_*\cO_X\to \cO_X$ is the identity map.
The Frobenius split varieties enjoy the following important property (see. e.g.
Proposition $1$ of \cite{MR}):
\begin{prop}
Let $X$ be a Frobenius split projective variety with a line bundle $\cL$ such that for some $i$ and
all large enough $m$  $H^i(X,\cL^m)=0$. Then $H^i(X,\cL)=0$.
\end{prop}

In order to prove Frobenius splitting of $\Fl^a_n$, we
use two statements from \cite{MR}, which we recall now. The first one is
Proposition $4$ of \cite{MR}:
\begin{prop}
Let $f:Z\to X$ be a proper morphism of algebraic varieties such that $f_*\cO_Z=\cO_X$.
Then if $Z$ is Frobenius split, then $X$ is also Frobenius split.
\end{prop}

\begin{cor}
If $R_n$ is Frobenius split, then $\Fl_n^a$ is Frobenius split as well.
\end{cor}
\begin{proof}
The normality of $\Fl^a_n$ implies $\pi_{n*}\cO_{R_n}=\cO_{\Fl_n^a}$.
\end{proof}

In order to prove that $R_n$ is Frobenius split we use the Mehta-Ramanathan
theorem (Proposition $8$ of \cite{MR}) which we recall now:
\begin{thm}\label{MR}
Let $Z$ be a smooth projective variety of dimension $M$ and let $Z_1, \dots,Z_M$ be codimension one
subvarieties satisfying the following conditions:
\begin{enumerate}
\item For any $I\subset\{1,\dots,M\}$ the intersection $\cap_{i\in I} Z_i$
is smooth of codimension $\#I$.
\item There exists a global section $s$ of the anti-canonical bundle $K^{-1}$ on $Z$ such that the zero divisor of $s$ equals $Z_1+\dots + Z_M+D$
for some effective divisor $D$ with $\cap_{i=1}^M Z_i\notin supp D$.
\end{enumerate}
Then $Z$ is Frobenius split and for any subset $I\subset \{1,\dots,M\}$ the intersection
$Z_I=\cap_{i\in I} Z_i$ is Frobenius split as well.
\end{thm}

In our situation $Z=R_n$ over a field ${\mathsf k}=\overline{\mathbb F}_p$
and $M=n(n-1)/2$ is the number of positive roots.
Let us construct the divisors $Z_1,\dots,Z_M$. For convenience, we denote them by
$Z_{i,j}$, $1\le i\le j\le n-1$. Recall that we have a tower of successive $\bP^1$-fibrations
$\rho_l:R_n(l)\to R_n(l-1)$ such that $R_n(M)=R_n$. For each $l$ we construct
a section $s_l$ of $\rho_l$ as follows. We note that in order to specify an element in the fiber
$\rho_l^{-1} {\bf V}$ for some ${\bf V}\in R_n(l-1)$ it suffices to determine the space
$(\rho_l({\bf V}))_{i,j}$, where $\beta_l=\al_{i,j}$. We consider three cases.
First, let $i=1$. Then we put
\[
(s_l({\bf V}))_{1,j}={\mathsf k} w_{j+1}.
\]
Second, let $j=n-1$. Then
\[
(s_l({\bf V}))_{i,n-1}={\mathsf k} w_n\oplus  {\mathsf k} w_{i-1}\oplus \dots \oplus {\mathsf k} w_1.
\]
Finally, let $i\ne 1$ and $j\ne n-1$. Then we set
\[
(s_l({\bf V}))_{i,j}=V_{i-1,j+1}\oplus {\mathsf k} w_{j+1}.
\]
It is  easy to check that with such a definition the resulting element belongs to $R_n(l)$.
In what follows we denote the image $s_l(R_n(l-1))$ by $s_l$ or by $s_{i,j}$ (recall $\beta_l=\al_{i,j}$).

Let $f_l=\rho_{l+1}\dots \rho_M: R_n\to R_n(l)$.
Define
\[
Z_l=Z_{i,j}=\{\bV\in R_n:\ f_l\bV\subset s_l\}.
\]
In other words, the divisor $Z_l$ can be constructed step by step
compatibly with the fibrations $\rho_\bullet$ in such a way that at the $l$-th step
one takes not the whole preimage, but the section $s_l$ only.

Let $\cL_{i,j}$, $1\le i\le j\le n-1$ be the $i$-dimensional  bundle on $R_n$ with
fiber $V_{i,j}$ at a point $\bV$.
We set $\om_{i,j}=\det^{-1} \cL_{i,j}$.
\begin{thm}\label{K-1}
We have
\[
K^{-1}_{R_n}=\cO\left(\sum_{l=1}^M Z_l\right)\T \bigotimes_{i=1}^{n-1} \om_{i,i}\T \bigotimes_{i=1}^{n-2} \om_{i,i+1}.
\]
\end{thm}

We first prove a lemma.
Let $B$ be a smooth projective variety and let $\cL_2$ be a two-dimensional bundle
on $B$ with a line subbundle $\cL_1$. Let $\rho:E\to B$ be a $\bP^1$-fibration with
$E=\bP(\cL_2)$. Let $s:B\to E$ be a section of $\rho$ defined by $\cL_1$.
In what follows we denote the section $s(B)\subset E$ simply by $s$.
\begin{lem}\label{proj}
For a line bundle $\cF$ on $E$ such that the restriction of $\cF$ to a fiber of $\rho$ is equal
to $\cO(k)$ one has
\[
\cF=\cO(ks)\T \rho^*(\cF|_{s}\T \cO(-ks)|_{s}).
\]
\end{lem}
\begin{proof}
We note that $\cF\T \cO(-ks)$ restricts trivially to a fiber of $\rho$ and therefore can be
pulled back from some line bundle on the section.
\end{proof}

We apply this lemma to the case $B=R_n(l-1)$, $E=R_n(l)$,  $\rho=\rho_l$, $s$
being a section constructed above. The bundles $\cL_1$ and $\cL_2$ in our situation are
described as follows: the fiber of $\cL_1$ at a point $\bV\in R_n(l-1)$ is equal to
\begin{equation}\label{L1}
\frac{V_{i-1,j+1}\oplus {\mathsf k}w_{j+1}}{V_{i-1,j}}
\end{equation}
and the fiber of $\cL_2$ at a point $\bV$ is equal to
\begin{equation}\label{L2}
\frac{V_{i,j+1}\oplus {\mathsf k}w_{j+1}}{V_{i-1,j}}.
\end{equation}
For $\cF$ we first take $K^{-1}_{R_n(l)}$ and then $\om_{i,j}$, where $\beta_l=\al_{i,j}$.
\begin{lem}\label{K}
Let $i\ne 1$ and $j\ne n-1$. Then
\[
K_E^{-1}=\cO(2s_{i,j})\T \rho^*(K_B^{-1})\T \rho^*(\om_{i,j+1}\T (\om_{i-1,j+1}^*)^{\T 2}\T \om_{i-1,j}).
\]
Let $i=1$. Then
\[
K_E^{-1}=\cO(2s_{1,j})\T \rho^*(K_B^{-1})\T \rho^*(\om_{1,j+1}).
\]
Let $j=n-1$. Then
\[
K_E^{-1}=\cO(2s_{i,n-1})\T \rho^*(K_B^{-1})\T \rho^*(\om_{i-1,n-1}).
\]
\end{lem}
\begin{proof}
We prove the first formula (the rest of the proof is very similar).
Our main tool is Lemma \ref{proj}.
Let $s=s_{i,j}$.
We note that the restriction of $K_E$ to the fibers of the map $\rho$ equals
$\cO(-2)$.
Also
\[
\cO(s)|_s\simeq Hom(\cL_1,\cL_2/\cL_1)\simeq T_{E/B},
\]
where $T_{E/B}$ is the normal line bundle to $s\simeq B$.
Consider the exact sequence
\[
0\to T_{E/B}\to T_E\to \rho^* T_B\to 0.
\]
Since $K_E=\det T_E^*$, we obtain $K_E=\det T_{E/B}^*\T \det \rho^* T^*_B$.
Therefore, Lemma \ref{proj} gives
\[
K_E^{-1}=\cO(2s)\T \rho^* (K_B^{-1})\T \rho^* (T_{E/B}^*).
\]
Now explicit computation of
$T_{E/B}=\cL_1^*\otimes (\cL_2/\cL_1)$ (using \eqref{L1} and \eqref{L2}) gives the desired formula.
\end{proof}

We now take $\cF=\om_{i,j}$ (recall $\beta_l=\al_{i,j}$ and
$\rho=\rho_l:R_n(l)\to R_n(l-1)$).
\begin{lem}\label{om}
Let $i\ne 1$ and $j\ne n-1$. Then
\[
\om_{i,j}=\cO(s_{i,j})\T \rho^*(\om_{i,j+1}\T \om_{i-1,j+1}^*\T \om_{i-1,j}).
\]
Let $i=1$. Then
\[
\om_{1,j}=\cO(s_{1,j})\T \rho^*(\om_{1,j+1}).
\]
Let $j=n-1$. Then
\[
\om_{i,n-1}=\cO(s_{i,n-1})\T \rho^*(\om_{i-1,n-1}).
\]
\end{lem}
\begin{proof}
We note that the restriction of $\om_{i,j}$ to the fibers of $\rho$ equals to $\cO(1)$.
Now the formula can be proved by an explicit computation using
Lemma~\ref{proj}.
\end{proof}

\begin{cor}\label{lift}
$K_{R_n}^{-1}=\bigotimes_{i=1}^{n-1} \om_{i,i}^{\T 2}$.
\end{cor}
\begin{proof}
We substitute the expression for $\cO(s_{i,j})$ in terms of $\om_{i,j}$ from Lemma \ref{om}
into the formulas from Lemma \ref{K}.
\end{proof}

\begin{cor}
Theorem \ref{K-1} holds.
\end{cor}
\begin{proof}
Recall the $(\bP^1)^m$-fibrations $\bar\rho_m:R_n(m(m+1)/2)\to R_n(m(m-1)/2)$, where $n-1\ge m\ge 1$.
Using Lemmas \ref{K} and \ref{om} we obtain
\[
K^{-1}_{R_n}=\bigotimes_{i=1}^{n-1} \cO(Z_{i,i})\T \bigotimes_{i=1}^{n-1} \om_{i,i}
\T \bigotimes_{i=1}^{n-3} \om_{i,i+2}^*\T \bar\rho^*_{n-1} K^{-1}_{R_n((n-1)(n-2)/2)}.
\]
Using Corollary \ref{indcor} and Lemmas \ref{K} and \ref{om} again, we rewrite further
\begin{multline*}
K^{-1}_{R_n}=\bigotimes_{i=1}^{n-1} \cO(Z_{i,i})\T \bigotimes_{i=1}^{n-2} \cO(Z_{i,i+1})\T\\
\bigotimes_{i=1}^{n-1} \om_{i,i}\T \bigotimes_{i=1}^{n-2} \om_{i,i+1}
\T \bigotimes_{i=1}^{n-4} \om_{i,i+3}^*\T \bar\rho^*_{n-2}\bar\rho^*_{n-1} K^{-1}_{R_n((n-2)(n-3)/2)}.
\end{multline*}
Continuing further, we arrive at the desired formula.
\end{proof}

\begin{cor}
The varieties $R_n$ and $\Fl_n^a$ are Frobenius split.
\end{cor}
\begin{proof}
According to Theorem \ref{MR} it suffices to find a section of the line bundle
\[
\bigotimes_{i=1}^{n-1} \om_{i,i}\T \bigotimes_{i=1}^{n-2} \om_{i,i+1}.
\]
which does not vanish at the point $\cap_{l=1}^M Z_l$. But it is easy to see that
this line bundle does not have any base points at all.
\end{proof}

\begin{thm}
All the degenerate partial flag varieties $\Fl_\bd^a$ are Frobenius split.
\end{thm}
\begin{proof}
We note that the resolution $R_\bd$ can be realized inside $R_n$ as an intersection
$\bigcap_{(i,j)\notin P_\bd} Z_{i,j}$. Therefore, Theorem \ref{MR} guaranties the Frobenius splitting
for $R_\bd$. Now the normality of $\Fl_\bd^a$ implies the desired
Frobenius splitting for $\Fl_\bd^a$.
\end{proof}

We close this section with the following remark.
The divisors $Z_{i,j}$ produce a cell decomposition for $R_n$.
Namely, introduce the following notations:  $Z_{i,j}=Z_{\beta}$ if
$\beta=\al_{i,j}$; and for a subset
$I\subset R_+$: $Z_I=\bigcap_{\beta\in I} Z_\beta$. We set
$\overset{\circ}{Z}_I:=Z_I\setminus \left( \bigcup_{J\supsetneqq I}Z_J\right)$.
Then we have a cell decomposition
\[
R_n=\bigsqcup_{I\subset R^+}\overset{\circ}{Z}_I.
\]
In general, these cells are different from  the cells of $R_n$ we are using in this paper
(however we conjecture that the codimension one cells do coincide).
For example, the image $\pi_n(\overset{\circ}{Z}_I)$ is not always a cell.  The first example
is $\msl_4$, $I=\{\al_{1,1},\al_{3,3},\al_{1,2},\al_{2,3}\}$. 
In this case $\pi_n(\overset{\circ}{Z}_I)\simeq \bP^1$.

\section{The BW-type theorem and graded character formula}

\subsection{Rational singularities}
We prove that the varieties $\Fl^a_\bd$ over $\overline{\mathbb F}_p$ and over $\bC$ have rational singularities.
Recall the desingularization $Y_\bd$ introduced in the proof of Proposition \ref{Yd}.
\begin{lem}
The variety $\Fl^a_\bd$ is Gorenstein, i.e. the dualizing complex
$K_{\Fl^a_\bd}$ is a line bundle. The resolution $\tau_\bd:\ Y_\bd\to\Fl^a_\bd$
is crepant, i.e. $K_{Y_\bd}=\tau_\bd^* K_{\Fl^a_\bd}$.
\end{lem}
\begin{proof}
We know that $\Fl^a_\bd$ is a locally complete intersection. By the adjunction
formula, it follows that $\Fl^a_\bd$ is Gorenstein. According to the proof
of~Proposition~\ref{Yd}, the map $\tau_\bd:Y_\bd\to \Fl^a_\bd$
is one-to-one off codimension two in $Y_\bd$. Hence, the canonical line bundle
$K_{Y_\bd}$ coincides with $\tau_\bd^*K_{\Fl^a_\bd}$ off codimension two. Hence
the desired equality $K_{Y_\bd}=\tau_\bd^* K_{\Fl^a_\bd}$.
\end{proof}

\begin{rem}
Corollary \ref{lift} says that the canonical line bundle of the complete
degenerate flag variety is given by
$K_{\Fl^a_\bd}=\prod_{i=1}^{n-1} (\om_{i,i}^*)^{\T 2}$.
\end{rem}

\begin{thm}
\label{ssylka}
For the projection $\tau_\bd:\ Y_\bd\to\Fl_\bd^a$ we have a canonical
isomorphism $R(\tau_\bd)_*\cO=\cO$, i.e. $(\tau_\bd)_*\cO=\cO$ and
$R^i(\tau_\bd)_*\cO=0$ for all $i>0$.
\end{thm}
\begin{proof}
We know that $\Fl_\bd^a$ is normal, so that $(\tau_\bd)_*\cO=\cO$.
Recall the Grauert--Riemenschneider vanishing theorem (\cite{GR}):
\begin{equation}\label{GR}
R^i(\tau_\bd)_* K_{Y_\bd}=0 \text { for all } i>0.
\end{equation}
This theorem holds for all varieties over $\bC$, but not over $\overline{\mathbb F}_p$. 
However, in Lemma \ref{GRp} we show that \eqref{GR} is true over $\overline{\mathbb F}_p$.    
Since $K_{Y_\bd}=\tau_\bd^* K_{\Fl^a_\bd}$ is the pull-back of a line bundle, the projection formula says
\[
R^i(\tau_\bd)_*(\tau_\bd^* \cL)=(\cL\T K^{-1}_{\Fl^a_\bd})\T R^i(\tau_\bd)_* K_{Y_\bd}=0
\]
for any line bundle $\cL$ on $\Fl^a_\bd$. Using the projection formula again, we arrive at the
desired vanishing of higher direct images of $\cO$.
\end{proof}

The following lemma is due to the anonymous referee.
\begin{lem}\label{GRp}
Formula \eqref{GR} is true over $\overline{\mathbb F}_p$.
\end{lem}
\begin{proof}
Theorem \ref{MR} together with Theorem 1.3.14 from \cite{BK} imply \eqref{GR} over
$\overline{\mathbb F}_p$ provided that $\tau_\bd$ is one-to-one outside the union of all divisors
$Z_{i,j}$. So our goal is to show that if $\bV\in Y_\bd$ satisfy $\bV\notin \bigcup_{i,j} Z_{i,j}$,
then all the entries $V_{k,l}$ are determined by the diagonal subspaces $V_{k,k}$. Since 
$pr_{l+1} V_{k,l}\subset V_{k,l+1}$ it suffices to prove that $w_{l+1}\notin V_{k,l}$ for all $k,l$.
Assume that for some $k,l$ we have $w_{l+1}\in V_{k,l}$. Let $k_0$ be the smallest number such that
$w_{l+1}\in V_{k_0,l}$. If $k_0=1$, then $\bV\in Z_{1,l}$. Now let $k_0>1$. Then 
$pr_{l+1} V_{k_0-1,l}=V_{k_0-1,l+1}$ (since $w_{l+1}\notin V_{k_0-1,l}$). Therefore 
we arrive at 
\[
V_{k_0,l}=V_{k_0-1,l}\oplus \spa(w_{l+1})=V_{k_0-1,l+1}\oplus \spa(w_{l+1})
\]     
and thus $\bV\in Z_{k_0,l}$.
\end{proof}

\subsection{The BW-type theorem}
In this subsection we prove an analogue of the Borel-Weil theorem. 
Let $\la$ be a dominant integral weight.
Consider the map $\imath_\la: \Fl^a_n\to \bP(V_\la^a)$. Define a line bundle
$\cL_\la=\imath_\la^* \cO(1)$ on $\Fl^a_n$.

\begin{prop}\label{van} We have
\[
H^{>0}(\Fl^a_n,\cL_\la)=H^{>0}(R_n,\pi_n^*\cL_\la)=0.
\]
\end{prop}
\begin{proof}
We give two proofs here. The first short one is due to the referee and uses deep results 
on Frobenius splitting from \cite{BK}. The second one uses partial degenerate flag varieties
and their desingularizations.

Recall (see \cite{BK}, Definition 1.4.1) that a scheme $X$ is called Frobenius split relative 
to a Cartier divisor $D$ (or simply $D$-split) if there exists a $\cO_X$-linear map
$\psi: F_*(\cO_X(D))\to \cO_X$ such that for a canonical section $\sigma$ of 
$\cO_X(D)$ the composition 
\[
\phi=\psi\circ F_*(\sigma):F_*\cO_X\to \cO_X
\] 
is a Frobenius splitting for $X$ (i.e. the composition 
$\cO_X\to F_*\cO_X\stackrel{\phi}{\longrightarrow}\cO_X$ is the identity map). 
Combining Theorem \ref{MR}, Theorem  \ref{K-1}
and \cite{BK}, Proposition 1.4.12,  we obtain that $R_n$ is 
$(p-1)(D_1+D_2)$-split, where 
$D_1$ is the divisor of zeroes of $\bigotimes_{i=1}^{n-1} \om_{i,i}$ and $D_2$ is the divisor of zeroes of
$\bigotimes_{i=1}^{n-2} \om_{i,i+1}$. Using \cite{BK}, Remark 1.4.2, (ii) we obtain that $R_n$ is
$D_1$-split as well. Since $D_1$ is pulled back from a line bundle on $\Fl^a_n$, our Lemma follows from
\cite{BK}, Theorem 1.4.8 (i).  

Now let us give the second proof.
We note that since $\Fl^a_n$ has rational singularities,
the equalities
\[\
H^k(\Fl^a_n,\cL_\la)\simeq H^k(R_n,\pi_n^*\cL_\la)
\]
hold for all $k\ge 0$.
Now assume
$\la$ is regular. Then since the map  $\Fl^a_n\to \bP(V_\la^a)$ is an embedding, the line bundle
$\cL_\la$ is very ample. Therefore, for any $k$ and big enough $N$ one has
$H^k(\Fl_n^a,\cL_\la^{\T N})=0$. This implies $H^k(\Fl_n^a,\cL_\la)=0$, because
$\Fl^a_n$ is Frobenius split over $\overline{\mathbb F}_p$ for any $p$.
For a non regular $\la$, let $\Fl^a_\bd$ be the corresponding
degenerate parabolic
flag variety, which is embedded into $\bP(V_\la^a)$. Then we have the following
commutative diagram of projections:
$$
\begin{CD}
R_n @>\eta>> R_\bd\\
@V\pi_n VV  @VV\pi_\bd V\\ 
\Fl_n^a @>\mu>> \Fl^a_\bd
\end{CD}
$$
Let $\cL'_\la$ be a line bundle on $\Fl^a_\bd$  which is the pull back of the bundle $\cO(1)$
on $\bP(V_\la^a)$. Then $\cL_\la=\mu^*\cL'_\la$.
Since $\cL'_\la$ is very ample, and $\Fl^a_\bd$ is Frobenius split over
$\overline{\mathbb F}_p$ for any $p$, $H^k(\Fl^a_\bd,\cL'_\la)=0$
(for positive $k$). Since $\Fl^a_\bd$ has rational singularities,
$H^k(R_\bd,\pi_\bd^*\cL'_\la)=H^k(\Fl^a_\bd,\cL'_\la)(=0$ for positive $k$).
Now since $\eta$ is a fibration with the fibers being towers of successive
$\bP^1$-fibrations, we obtain $H^k(R_n,\eta^*\pi_\bd^*\cL'_\la)=
H^k(R_\bd,\pi_\bd^*\cL'_\la)(=0$ for positive $k$). Finally, since $\Fl_n^a$
has rational singularities,
and $\eta^*\pi_\bd^*\cL'_\la=\pi_n^*\cL_\la$, we arrive at $H^k(\Fl_n^a,\cL_\la)=
H^k(R_n,\pi_n^*\cL_\la)=H^k(R_n,\eta^*\pi_\bd^*{\cL'}_\la)(=0$ for $k>0$).
\end{proof}

\begin{prop}
\label{ivan}
We have
$H^0(\Fl^a_n,\cL_\la)^*\simeq H^0(R_n,\pi_n^*\cL_\la)^*\simeq V_\la^a$.
\end{prop}
\begin{proof}
We note that there exists an embedding $(V_\la^a)^*\hk H^0(\Fl^a_n,\cL_\la)$.
In fact take an element $v\in (V_\la^a)^*\simeq H^0(\bP(V_\la^a),\cO(1))$.
Then restricting to the embedded variety $\Fl^a_n$ we obtain a section of $\cL_\la$. Assume that
it is zero. Then $v$ vanishes on the open cell $(N^-)^a\cdot \bC v_\la$. But
the linear span of the elements of this cell coincides with the whole representation $V_\la^a$.
Therefore, the restriction map  $(V_\la^a)^*\to H^0(\Fl^a_n,\cL_\la)$ is an embedding.

We recall that the varieties $\Fl_n^a$ are flat degenerations of the classical flag
varieties. Since the higher cohomology of $\cL_\la$ vanish 
(see Proposition~\ref{van}), 
we arrive at the equality of the dimensions
of $H^0(\Fl_n^a,\cL_\la)$ and of $V_\la$.
Therefore, the embedding  $(V_\la^a)^*\to H^0(\Fl^a_n,\cL_\la)$ is an isomorphism.
\end{proof}

Combining Propositions~\ref{van} and~\ref{ivan} we obtain the analogue of
the Borel-Weil theorem for degenerate flags: 

\begin{thm}
\label{divan}
We have
\begin{gather*}
H^0(\Fl^a_n,\cL_\la)^*\simeq H^0(R_n,\pi_n^*\cL_\la)^*\simeq V_\la^a,\\
H^{>0}(\Fl^a_n,\cL_\la)=H^{>0}(R_n,\pi_n^*\cL_\la)=0.
\end{gather*}
\end{thm}

Similarly one proves a parabolic version of the BW-type theorem:
\begin{thm}
Let $\la$ be a $\bd$-dominant weight, i.e. $(\la,\om_d)>0$ implies $d\in\bd$.
Then there exists a map $\imath_\la:\Fl^a_\bd\to\bP(V_\la^a)$.
We have
$$
H^0(\Fl^a_\bd,\imath_\la^*\cO(1))^*\simeq V_\la^a,\
H^{>0}((\Fl^a_\bd,\imath_\la^*\cO(1)))=0.
$$
\end{thm}

\subsection{The $q$-character formula}
We now compute the $q$-character (PBW-graded character) of the modules $V_\la^a$
(for combinatorial formula see \cite{FFoL1}). 
For this we use the Atiyah-Bott-Lefschetz fixed points formula applied to the variety $R_n$
(so our formula is an analogue of the Demazure character formula).
Recall that the $T$-fixed points on $R_n$ are labeled by the admissible collections $\bS=(S_{i,j})$,
i.e. those satisfying
$S_{i,j}\subset \{1,\dots,i,j+1,\dots,n\}$, $\# S_{i,j}=i$ and
\begin{equation}\label{dih}
S_{i,j}\subset S_{i+1,j}\subset S_{i+1,j+1}\cup\{j+1\}.
\end{equation}
In order to state the theorem we prepare some notations.
Assume that we have fixed the sets $S_{i-1,j}$ and $S_{i,j+1}$. Then condition \eqref{dih}
says that there exist exactly two variants for $S_{i,j}$,
namely
\[
S_{i,j}=S_{i-1,j}\cup \{a\} \text{ or } S_{i,j}=S_{i-1,j}\cup \{b\},
\]
where $\{a,b\}=S_{i,j+1}\cup\{j+1\}\setminus S_{i-1,j}$. Given a collection $\bS$ we denote
the numbers $a,b$ as above by $a^\bs_{i,j}$ and $b^\bs_{i,j}$. We have:
\[
S_{i,j}=S_{i-1,j}\cup \{a^\bs_{i,j}\},\ \ S_{i,j+1}\setminus S_{i-1,j}=\{a^\bs_{i,j},b^\bs_{i,j}\}.
\]
We denote by $S'_{i,j}$ the set $(S_{i,j}\setminus \{a^\bs_{i,j}\})\cup \{b^\bs_{i,j}\}$.
\begin{example}
Let $n=3$, $S_{1,1}=(2)$, $S_{1,2}=(1)$ and $S_{2,2}=(1,3)$. Then
\[
a^\bS_{1,1}=2, a^\bS_{1,2}=1, a^\bS_{2,2}=3 \text{ and }
b^\bS_{1,1}=1, b^\bS_{1,2}=3, b^\bS_{2,2}=2.
\]
\end{example}

Recall that the variety $R_n$ sits inside the product of Grassmann varieties
$\prod_{1\le i\le j<n} Gr(i,W_{i,j})$.
Each $\bigwedge^i(W_{i,j})$ is acted upon by $\g^a\oplus\bC d$
and therefore each
Grassmannian  carries a natural action of the group $G^a\rtimes\bC^* $, where the additional $\bC^*$ part
corresponds to the PBW-grading operator. So we have an $n$-dimensional torus $T\rtimes \bC^*$
acting on $Gr(i,W_{i,j})$. A $T$-fixed point $p(\bS)\in R_n$ is a product of the fixed points
$p(S_{i,j})\in Gr(i,W_{i,j})$. We denote by $\gamma(S_{i,j})\in\h^*\oplus\bC d$ the (extended)
weight of the vector $p(S_{i,j})\in W_{i,j}$.
Explicitly, let $S_{i,j}=(l_1,\dots,l_i)$. Then
\[
\gamma(S_{i,j})=(\om_{l_1}-\om_{l_1-1})+\dots +(\om_{l_i}-\om_{l_i-1}) + \#\{r:\ l_r>i\}d.
\]
(here $\om_0=\om_n=0$).
For an element
$$\gamma=m_1\om_1+\dots +m_{n-1}\om_{n-1}+md^*\in \h^*\oplus\bC d^*$$
we denote by $e^\gamma$ the element $(e^{\om_1})^{m_1}\dots (e^{\om_{n-1}})^{m_{n-1}} q^m$
in the group algebra (so, $q=e^{d^*}$).
\begin{example}
Let $z_1=e^{\om_1}, \dots, z_{n-1}=e^{\om_{n-1}}$. Then for $S_{i,j}=(l_1,\dots,l_i)$
\[
e^{\gamma(S_{i,j})}=z_{l_1}z_{l_1-1}^{-1}\dots z_{l_i}z_{l_i-1}^{-1} q^{\#\{r:\ l_r>i\}}.
\]
In particular, for $n=3$, $i=j=1$ we have
\begin{gather*}
S_{1,1}=(1):\ \ \gamma(1,1)=\om_1,\ e^{\gamma(1,1)}=z_1;\\
S_{1,1}=(2):\ \ \gamma(1,1)=\om_2-\om_1+d,\ e^{\gamma(1,1)}=z_1^{-1}z_2q;\\
S_{1,1}=(3):\ \ \gamma(1,1)=-\om_2 +d,\ e^{\gamma(1,1)}=z_2^{-1}q.
\end{gather*}
\end{example}

We need one more piece of notations to formulate the theorem.  
Let $\imath_\la:\Fl^a_{n}\to \bP(V_\la^a)$ be the standard map (which is an embedding for regular $\la$).
We denote by $\gamma_\la(\bS)$ the (extended) weight of $\imath_\la(p(\bS))$ (note that this weight
depends only on the diagonal entries $S_{i,i}$).
In other words, $\gamma_\la(\bS)=\sum_{i=1}^{n-1}\ell_i\gamma(S_{i,i})$
where $\la=\sum_{i=1}^{n-1}\ell_i\omega_i$.

\begin{thm}
The $q$-character of the representation $V^a_\la$ is given by the sum over all
admissible collections $\bS$ of the summands
\begin{equation}\label{AB}
\frac{e^{\gamma_\la(\bS)}}
{\prod_{1\le i\le j<n} \left(1-e^{\gamma(S'_{i,j})}e^{-\gamma(S_{i,j})}\right)}.
\end{equation}
\end{thm}
\begin{proof}
Recall the Atiyah-Bott-Lefschetz formula (see \cite{AB}, \cite{T}): let $X$ be a smooth projective algebraic
$M$-dimensional variety and let $\cL$ be a line bundle on $X$. Let $T$ be an algebraic torus acting on $X$ with a finite
set $F$ of fixed points. Assume further that $\cL$ is $T$-equivariant.
Then for each $p\in F$ the fiber $\cL_p$ is $T$-stable. We note also that since $p\in F$, the tangent space
$T_pX$ carries a natural $T$-action. Let $\gamma_1^p,\dots,\gamma_M^p$ be the
weights of the eigenvectors of $T$-action on $T_pX$.
Then the Atiyah-Bott-Lefschetz formula gives the following expression
for the character of the Euler characteristics:
\begin{equation}\label{ABL}
\sum_{k\ge 0} (-1)^k \ch H^k(X,\cL)=\sum_{p\in F} \frac{\ch \cL_p}{\prod_{l=1}^M (1-e^{-\gamma_l^p})}.
\end{equation}
We apply this formula in our situation: 
$X=R_{n}$, $\cL=\pi_{n}^* \cL_\la$ with the action of the extended 
torus $T\rtimes \bC^*$. 
Since $H^{>0}(R_n,\pi_n^*\cL_\la)=0$, the Euler
characteristics coincides with the character of the zeroth cohomology, i.e. with the character of
$(V_\la^a)^*$. Therefore, for each admissible $\bS$ we need to compute the character of 
$\pi_{n}^* \cL_\la$ at $p(\bS)$ and the eigenvalues of the torus action in $T_{p(\bs)}R_n$.
Further, the sum in \eqref{ABL}
runs over the set of $T$-fixed points in $R_n$ and for each summand the numerator in the $\bS$-th term is exactly the character of the dual 
line $\left(\imath_\la \pi_n p(\bS)\right)^*$, 
which equals to $e^{-\gamma_\la(\bS)}$ (the minus sign comes from the fact that
$\cL_\la=\imath_\la^*\cO(1)$ and a fiber of $\cO(1)$ is the dual line).
It only remains to compute the torus action in the tangent space $T_{p(\bS)}R_n$.

Recall that $R_n$ is a tower of successive $\bP^1$-fibrations $R_n(l)\to R_n(l-1)$. Fix an admissible
$\bS$. Then the surjections $R_n\to R_n(l)$ define the $T$-fixed points $p(\bS(l))$ in each $R_n(l)$
(note that $\bS(l)$ consists of $S_{i,j}$ such that for $\beta_k=\al_{i,j}$ one has $k\le l$).
For each $l=1,\dots,M$ we denote by $v_l\in T_{p(\bS(l))}R_n(l)$ a tangent vector to the fiber
of the map $R_n(l)\to R_n(l-1)$ at the point $p(\bS(l-1))$.
Then it is easy to see that the weights of the eigenvectors of the $T$ action in $T_{p(\bs)}R_n$
are exactly the weights of the vectors $v_l$, $l=1,\dots,M$.

So let us fix an $l$,  $1\le l\le M$ and $i,j$ with  $\al_{i,j}=\beta_l$.
Let us denote by $Y_l$ the set of all pairs $(t,u)$ such that for the root $\al_{t,u}=\beta_r$ one has $r\le l$.
Then the fiber $\bP^1$ of the map $R_n(l)\to R_n(l-1)$
at the point $p(\bS(l-1))$ consists of all collections $(V_{t,u})$ with $(t,u)\in Y_l$ subject to the following
conditions:
\begin{itemize}
\item $V_{t,u}=p(S_{t,u})$ if $\al_{t,u}\ne \beta_l$,\\
\item $V_{i,j}\supset p(S_{i-1,j})$,\\
\item $V_{i,j}\subset p(S_{i-1,j})\oplus\bC w_{a^\bS_{i,j}}\oplus\bC w_{b^\bS_{i,j}}$.
\end{itemize}
Now it is easy to see that the character of the tangent vector to this fiber at the point
$p(\bS(l-1))$ is equal to $e^{\gamma(S'_{i,j})}e^{\gamma(S_{i,j})^{-1}}$
(recall $a^\bS_{i,j}\in S_{i,j}$ and $S'_{i,j}=(S_{i,j}\setminus \{a^\bS_{i,j}\}) \cup \{b^\bS_{i,j}\}$).
\end{proof}

\begin{rem}
We note that the Euler characteristics 
$$\sum_{k\ge 0} (-1)^k \ch H^k(R_{n},\pi_{n}^*\cL_\la)$$
is equal to $\ch (V_\la^a)^*$. But in each summand \eqref{AB} both numerator and denominator
differ from the corresponding summand in the Atiyah-Bott-Lefschetz formula \eqref{ABL} by the changed 
of variables $z_i\to z_i^{-1}$ and $q\to q^{-1}$. Via this change we pass from the character of
$(V_\la^a)^*$ to the character of $V_\la^a$. 
\end{rem}

\begin{example}
Let $n=2$. Then the formula above says
\[
\ch_q V_{m\om}=\frac{z^m}{1-qz^{-2}}+\frac{z^{-m}q^m}{1-q^{-1}z^2}=z^m + qz^{m-2} +\dots + q^mz^{-m}.
\]
\end{example}
\
\begin{example}
Let $n=3$. Then the contribution of a fixed point with $S_{1,1}=(2)$, $S_{1,2}=(1)$, $S_{2,2}=(1,3)$  is given by
\[
\frac{q^2}{(1-z_1^{-1}z_2^{-1}q)(1-z_1^{2}z_2^{-1}q)(1-z_1^{-1}z_2^{2}q)}
\]
and the contribution of a fixed point with $S_{1,1}=(2)$, $S_{1,2}=(3)$, $S_{2,2}=(1,3)$  is given by
\[
\frac{q^2}{(1-z_1z_2q^{-1})(1-z_1z_2^{-2})(1-z_1^{-2}z_2)}.
\]
We note that these are exactly the points which are mapped by $\pi_3$ to the only singular point of
$\Fl^a_3$, which is  torus fixed and labeled by $S_{1,1}=(2)$, $S_{2,2}=(1,3)$.
\end{example}

\section*{Acknowledgments}
We are grateful to Shrawan Kumar, Alexander Kuznetsov, and Peter Littelmann
for useful discussions.
We are also grateful to R.~Bezrukavnikov and D.~Kazhdan for organizing
the 15th Midrasha Mathematicae ``Derived Categories of Algebro-Geometric Origin
and Integrable Systems'' at IAS at the Hebrew University of Jerusalem where
this work was conceived.
This paper was written during the E.~F. stay at the Hausdorff Research Institute for Mathematics.
The hospitality and perfect working conditions of the Institute are gratefully
acknowledged.
The work of Evgeny Feigin was partially supported
by the Russian President Grant MK-3312.2012.1, by the Dynasty Foundation,
by the AG Laboratory HSE, RF government grant, ag. 11.G34.31.0023, by the RFBR grants
12-01-00070, 12-01-00944 and by the Russian Ministry of Education and Science under the
grant 2012-1.1-12-000-1011-016.
M.~F. was partially supported by the RFBR grant 12-01-00944, the National Research University Higher School of Economics' Academic Fund award No.12-09-0062 and
the AG Laboratory HSE, RF government grant, ag. 11.G34.31.0023.
This study was carried out within the National Research University Higher School of Economics
Academic Fund Program in 2012-2013, research grant No. 11-01-0017.
This study comprises research findings from the ``Representation Theory
in Geometry and in Mathematical Physics" carried out within The
National Research University Higher School of Economics' Academic Fund Program
in 2012, grant No 12-05-0014.

\end{document}